\documentclass[11pt]{article}
\usepackage{latexsym,amsfonts,amsthm,amsmath,amscd,amssymb,color,url}
\usepackage[dvips]{graphicx}

\newtheorem{lemma}{Lemma}[section]
\newtheorem{thm}[lemma]{Theorem}
\newtheorem{rem}[lemma]{Remark}
\newtheorem{prop}[lemma]{Proposition}

\newcommand\matZ{{\mathbb{Z}}}
\newcommand\matC{{\mathbb{C}}}

\newcommand\Sigmatil{{\widetilde\Sigma}}
\newcommand\gtil{{\widetilde{g}}}

\renewcommand{\hbar}{{\overline{h}}}

\newfont{\Got}{eufm10 scaled 1200}
\newcommand{\permu}{{\hbox{\Got S}}}

\newcommand{\compo}{\,{\scriptstyle\circ}\,}

\newcommand{\mycap} [1] {\caption{\footnotesize{#1}}}

\newcommand{\primo}{\textrm{I}}
\newcommand{\secon}{\textrm{I\!I}}
\newcommand{\terzo}{\textrm{I\!I\!I}}
\newcommand{\quart}{\textrm{I\!V}}

\newcommand{\sistema}[1]{\left\{\begin{array}{l} #1 \end{array}\right.}

\begin{document}

\title{Explicit computation of some\\ families of Hurwitz numbers, II}

\author{Carlo~\textsc{Petronio}\thanks{Partially supported by INdAM through GNSAGA, by
MIUR through PRIN ``Real and Complex Manifolds: Geometry, Topology and Harmonic Analysis''
and by UniPI through PRA 2018/22 ``Geometria e Topologia delle variet\`a''}}

\maketitle

\begin{abstract}\noindent
We continue our computation, using a combinatorial method based on Gronthendieck's \emph{dessins d'enfant},
of the number of (weak) equivalence classes of surface branched covers matching certain specific branch data.
In this note we concentrate on data with the surface of genus $g$ as source surface,
the sphere as target surface,
$3$ branching points, degree $2k$, and local degrees over the branching points
of the form $[2,\ldots,2]$, $[2h+1,3,2,\ldots,2]$, $\pi=\left[d_i\right]_{i=1}^\ell$.
We compute the corresponding (weak) Hurwitz numbers
for several values of $g$ and $h$, getting explicit arithmetic formulae in terms of the $d_i$'s.

\smallskip

\end{abstract}

\noindent
This paper is a continuation of~\cite{x1}, and it is based on the same methods, but the results
that we obtain here refer to a topologically more complex situation, so the required arguments
are more elaborate.
In this introduction we quickly review the subject matter, making the paper independent of~\cite{x1},
and we state our results.

\paragraph{Surface branched covers} A surface branched cover is a
map $$f:\Sigmatil\to\Sigma$$
where $\Sigmatil$ and $\Sigma$ are closed and connected surfaces and $f$ is locally modeled
on maps of the form
$$(\matC,0)\ni z\mapsto z^m\in(\matC,0).$$
If $m>1$ the point
$0$ in the target $\matC$ is called a \emph{branching point},
and $m$ is called the local degree at the point $0$ in the source $\matC$.
There are finitely many branching points, removing which, together
with their pre-images, one gets a genuine cover of some degree $d$.
If there are $n$ branching points, the local degrees at the points
in the pre-image of the $j$-th one form a partition $\pi_j$ of $d$ of some
length $\ell_j$, and the following Riemann-Hurwitz relation holds:
$$\chi\left(\Sigmatil\right)-(\ell_1+\ldots+\ell_n)=d\left(\chi\left(\Sigma\right)-n\right).$$
Let us now call \emph{branch datum} a 5-tuple
$$\left(\Sigmatil,\Sigma,d,n,\pi_1,\ldots,\pi_n\right)$$
and let us say it is \emph{compatible} if it satisfies the Riemann-Hurwitz relation.
(For a non-orientable $\Sigmatil$ and/or $\Sigma$ this relation
should actually be complemented with certain other necessary conditions,
but we restrict to an orientable $\Sigma$ in this paper, so we do not
spell out these conditions here.)

\paragraph{The Hurwitz problem}
The very old \emph{Hurwitz problem} asks which compatible branch data are
\emph{realizable} (namely, associated to some existing surface branched cover)
and which are \emph{exceptional} (non-realizable).
Several partial solutions
to this problem have been obtained over the time, and we quickly
mention here the fundamental~\cite{EKS}, the survey~\cite{Bologna}, and
the more recent~\cite{Pako, PaPe, PaPebis, CoPeZa, SongXu}.
In particular, for an orientable $\Sigma$ the problem has been shown
to have a positive solution whenever $\Sigma$ has positive genus.
When $\Sigma$ is the sphere $S$, many realizability and exceptionality
results have been obtained (some of experimental nature), but the general
pattern of what data are realizable remains elusive. One guiding
conjecture in this context is that \emph{a compatible branch datum is always
realizable if its degree is a prime number}. It was actually shown in~\cite{EKS}
that proving this conjecture in the special case of $3$ branching
points would imply the general case. This is why many efforts have
been devoted in recent years to investigating the realizability
of compatible branch data with base surface $\Sigma$ the sphere $S$ and having $n=3$
branching points. See in particluar~\cite{PaPe, PaPebis} for some evidence
supporting the conjecture.

\paragraph{Hurwitz numbers}
Two branched covers
$$f_1:\Sigmatil\to\Sigma\qquad f_2:\Sigmatil\to\Sigma$$
are said to be \emph{weakly equivalent} if there exist homeomorphisms $\gtil:\Sigmatil\to\Sigmatil$
and $g:\Sigma\to\Sigma$
such that $f_1\compo\gtil=g\compo f_2$, and \emph{strongly equivalent} if
the set of branching points in $\Sigma$ is fixed once and forever and
one can take $g=\textrm{id}_\Sigma$.
The \emph{(weak or strong) Hurwitz number} of a compatible
branch datum is the number of (weak or strong) equivalence classes of branched covers
realizing it. So the Hurwitz problem can be rephrased as the question whether a
Hurwitz number is positive or not (a weak Hurwitz number can be smaller
than the corresponding strong one, but they can only vanish simultaneously).
Long ago Mednykh in~\cite{Medn1, Medn2} gave some formulae for the computation
of the strong Hurwitz numbers,
but the actual implementation of these formulae is rather elaborate in general.
Several results were also obtained in more recent years in~\cite{GKL, KM1, KM2, KML, MSS}.

\paragraph{Computations}
In this paper we consider branch data of the form
$$\leqno{(\heartsuit)}
\left(\Sigmatil,\Sigma=S,d=2k,n=3,
[2,\ldots,2],[2h+1,3,2,\ldots,2],\pi=\left[d_i\right]_{i=1}^\ell\right)$$
for $h\geqslant0$.  Here we employ square brackets to denote an unordered array
of integers with repetitions.
A direct calculation shows that such a datum is compatible
for $h\geqslant 2g-1$, where $g$ is the genus of $\Sigmatil$, and $\ell=h-2g+2$.
We compute the weak Hurwitz number of the datum
for $g=0,1,2$ and for some of the smallest possible $h$'s, namely in the following cases:
for $g=0$ and $h=0,1,2$;
for $g=1$ and $h=1,2$; for $g=2$ and $h=3$.
More values could be obtained using the same techniques as we
employ below, but the complication of the topological and
combinatorial situation grows very rapidly, and the arithmetic
formulae giving the weak Hurwitz numbers are likely to be rather intricate
for larger values of $g$ and/or $h$.

\medskip

We will denote by $T$ the torus and by $2T$ the genus-2 surface.

\begin{thm}\label{genus0:thm}
\begin{itemize}
\item $(g=0,\ h=0)$\quad The number of weakly inequivalent realizations of
$$(S,S,2k,3,[2,\ldots,2],[1,3,2,\ldots,2],\pi)$$
(with $\ell(\pi)=2$) is $0$ if $\pi$ contains $k$, and $1$ otherwise.
\item $(g=0,\ h=1)$\quad The number of weakly inequivalent realizations of
$$(S,S,2k,3,[2,\ldots,2],[3,3,2,\ldots,2],\pi)$$
(with $\ell(\pi)=3$) is $0$ if $\pi$ contains $k$, and $1$ otherwise.
\item $(g=0,\ h=2)$\quad The number $\nu$ of weakly inequivalent realizations of
$$(S,S,2k,3,[2,\ldots,2],[5,3,2,\ldots,2],\pi)$$
(with $\ell(\pi)=4$) is as follows:
\begin{itemize}
\item If $\pi=[p,p,p,p]$ or $\pi=[p,p,q,q]$ for distinct $p,q$ then $\nu=0$;
\item If $\pi=[p,p,p,q]$ for distinct $p,q$ then $\nu=0$ if $k$ is in $\pi$, and $\nu=1$ otherwise;
\item If $\pi=[p,p,q,r]$ for distinct $p,q,r$ then $\nu=1$ if $k$
is in $\pi$ or the sum of two entries of $\pi$, and $\nu=3$ otherwise;
\item If $\pi=[p,q,r,s]$ for distinct $p,q,r,s$  then $\nu=2$ if $k$
is the sum of two entries of $\pi$, while $\nu=3$ if $k$ is in $\pi$, and $\nu=6$ otherwise.
\end{itemize}
\end{itemize}
\end{thm}

\begin{thm}\label{genus1:thm}
\begin{itemize}
\item $(g=1,\ h=1)$\quad The number of weakly inequivalent realizations of
$$(T,S,2k,3,[2,\ldots,2],[3,3,2,\ldots,2],[2k])$$
is $\frac12 k(k-1)$.
\item $(g=1,\ h=2)$\quad The number of weakly inequivalent realizations of
$$(T,S,2k,3,[2,\ldots,2],[5,3,2,\ldots,2],[p,2k-p])$$
is $0$ for $p=k$,
otherwise it is
$$2\left[\frac14(k-p-1)^2\right]+\left[\frac p2\right]\cdot (k-p-1)+\left[\frac14(p-1)^2\right].$$
\end{itemize}
\end{thm}

\begin{thm}\label{genus2:thm}
$(g=2,\ h=3)$\quad
The number of weakly inequivalent realizations of
$$(2T,S,2k,3,[2,\ldots,2],[7,3,2,\ldots,2],[2k])$$
is given by
$$\frac1{48}
\left(7k^4 - 70k^3 + 290k^2 - 515k + 288\right) - \frac58(2k - 5)\left[\frac k2\right].$$
\end{thm}

\section{Weak Hurwitz numbers and dessins d'enfant}\label{DA:sec}
In this section we quickly recall the machinery described in~\cite{x1},
omitting all the (rather easy) proofs. Our techniques are based on the notion of
dessin d'enfant, popularized by Grothendieck in~\cite{Groth} (see also~\cite{Cohen}),
but actually known before his work and
already exploited to give partial answers to the Hurwitz problem (see~\cite{LZ, Bologna}
and the references quoted therein).
Here we explain how to employ the dessins d'enfant to compute weak Hurwitz numbers.
Let us fix until further notice a branch datum
$$\leqno{(\spadesuit)}
\left(\Sigmatil,\Sigma=S,d,n=3,
\pi_1=\left[d_{1i}\right]_{i=1}^{\ell_1},
\pi_2=\left[d_{2i}\right]_{i=1}^{\ell_2},
\pi_3=\left[d_{3i}\right]_{i=1}^{\ell_3}\right).$$
A graph $\Gamma$ is \emph{bipartite} if it has black and white
vertices, and each edge joins black to white. If $\Gamma$ is
embedded in $\Sigmatil$ we call \emph{region} a component $R$ of
$\Sigmatil\setminus\Gamma$, and
\emph{length} of $R$ the number of white (or black) vertices of
$\Gamma$ to which $R$ is incident (with multiplicity).
A pair $(\Gamma,\sigma)$ is called \emph{dessin d'enfant} representing $(\spadesuit)$
if $\sigma\in\permu_3$ and $\Gamma\subset\Sigmatil$ is a bipartite graph
such that:
\begin{itemize}
\item The black vertices of $\Gamma$ have valence $\pi_{\sigma(1)}$;
\item The white vertices of $\Gamma$ have valence $\pi_{\sigma(2)}$;
\item The regions of $\Gamma$ have length $\pi_{\sigma(3)}$.
\end{itemize}
We will also say that $\Gamma$ \emph{represents $(\spadesuit)$ through} $\sigma$.

\begin{rem}
\emph{Let $f:\Sigmatil\to S$ be a branched cover matching $(\spadesuit)$
and take $\sigma\in\permu_3$. If
$\alpha$ is a segment in $S$ with a black and a white end
at the branching points corresponding to $\pi_{\sigma(1)}$ and $\pi_{\sigma(2)}$, then
$\left(f^{-1}(\alpha),\sigma\right)$  represents
$(\spadesuit)$, with vertex colours of $f^{-1}(\alpha)$ lifted via $f$.}
\end{rem}

Reversing the construction described in the previous remark one gets the following:

\begin{prop}\label{from:Gamma:to:f:prop}
To a dessin d'enfant $(\Gamma,\sigma)$ representing $(\spadesuit)$
one can associate a branched cover $f:\Sigmatil\to S$
realizing $(\spadesuit)$, well-defined up to equivalence.
\end{prop}

We define the equivalence relation $\sim$ on dessins d'enfant generated by:
\begin{itemize}
\item $(\Gamma_1,\sigma_1)\sim(\Gamma_2,\sigma_2)$ if $\sigma_1=\sigma_2$ and
there is an automorphism $\gtil:\Sigmatil\to\Sigmatil$ such that
$\Gamma_1=\gtil\left(\Gamma_2\right)$ matching colours;
\item $(\Gamma_1,\sigma_1)\sim(\Gamma_2,\sigma_2)$ if $\sigma_1=\sigma_2\compo(1\,2)$ and
$\Gamma_1=\Gamma_2$ as a set but with vertex colours switched;
\item $(\Gamma_1,\sigma_1)\sim(\Gamma_2,\sigma_2)$ if $\sigma_1=\sigma_2\compo(2\,3)$ and
$\Gamma_1$ has the same black vertices as $\Gamma_2$ and
for each region $R$ of $\Gamma_2$ we have that $R\cap\Gamma_1$ consists
of one white vertex and disjoint edges joining this vertex with the black vertices
on the boundary of $R$.
\end{itemize}

\begin{thm}\label{equiv:Gamma:for:equiv:f:thm}
The branched covers associated as in Proposition~\ref{from:Gamma:to:f:prop}
to two dessins d'enfant are equivalent if and only if
the dessins are related by $\sim$.
\end{thm}

When the partitions $\pi_1,\pi_2,\pi_3$ in the branch datum $(\spadesuit)$
are pairwise distinct, to compute the corresponding weak Hurwitz number one can
stick to dessins d'enfant representing the datum through the identity, namely
one can list up to automorphisms of $\Sigmatil$ the bipartite graphs
with black and white vertices of valence $\pi_1$ and $\pi_2$ and
regions of length $\pi_3$. When the partitions are not distinct, however,
it is essential to take into account the other moves generating $\sim$.
In any case we will henceforth omit any reference to the permutations in $\permu_3$.

\paragraph{Relevant data and repeated partitions}
We now specialize again to a branch datum of the form $(\heartsuit)$.
We will compute its weak Hurwitz number $\nu$ by enumerating up to
automorphisms of $\Sigmatil$ the dessins d'enfant $\Gamma$
representing it through the identity, namely
the bipartite graphs $\Gamma$ with black vertices of
valence $[2,\ldots,2]$, the white vertices of valence
$[2h+1,3,2,\ldots,2]$, and the regions of
length $\pi$. Two remarks are in order:
\begin{itemize}
\item In all the pictures we will only draw the two white
vertices of $\Gamma$ of valence $(2h+1,3)$, and
we will decorate an edge of $\Gamma$ by an integer $a\geqslant1$
to understand that the edge contains $a$ black and $a-1$ white valence-2 vertices;
\item Enumerating these $\Gamma$'s up to automorphisms of $\Sigmatil$
already gives the right value of $\nu$ except if two of the
partitions of $d$ in coincide.
\end{itemize}

\begin{prop}
In a branch datum of the form $(\heartsuit)$ two of the partitions of $d$ coincide precisely in
the following cases:
\begin{itemize}
\item $g=0,\ h\geqslant 0,\ k=h+2$, with partions $[2,\ldots,2],[2k-3,3],[2,\ldots,2]$;
\item Any $g,\ h\geqslant 2g,\ k=2h+2-2g$, with partitions $[2,\ldots,2],[2h+1,3,2,\ldots,2],[2h+1,3,2,\ldots,2]$.
\end{itemize}
\end{prop}

\begin{proof}
The lengths of the partitions $\pi_1,\pi_2,\pi$ in $(\heartsuit)$ are $\ell_1=k$, $\ell_2=k-h$ and
$\ell=h+2-2g$. We can never have $\pi_1=\pi_2$. Since $k\geqslant h+2$ we can have $\ell_1=\ell$ only
if $g=0$ and $k=h+2$, whence the first listed item. We can have $\ell_2=\ell$ only for $k=2h+2-2g$,
whence $h\geqslant 2g$ and the data in the second listed item.
\end{proof}

This result implies that the data $(\heartsuit)$ relevant to Theorems~\ref{genus0:thm} to~\ref{genus2:thm} and
containing repetitions  are precisely
$$(S,S,4,3,[2,2],[1,3],[2,2])$$
$$(S,S,6,3,[2,2,2],[3,3],[2,2,2])$$
$$(S,S,8,3,[2,2,2,2],[5,3],[2,2,2,2])$$
$$(S,S,4,3,[2,2],[1,3],[1,3])$$
$$(S,S,8,3,[2,2,2,2],[3,3,2],[3,3,2])$$
$$(S,S,12,3,[2,2,2,2,2,2],[5,3,2,2],[5,3,2,2])$$
$$(T,S,8,3,[2,2,2,2],[5,3],[5,3]).$$
Moreover we have $\nu=0$
in the first and third cases by the \emph{very even data} criterion of~\cite{PaPebis},
while the other cases will be taken into account below.

\section{Genus 0}\label{genus0:sec}
In this section we prove Theorem~\ref{genus0:thm}, starting
from the very easy case $h=0$, for which there is only
one homeomorphism type of relevant graph and only one embedding
in $S$, as shown on the left in Fig.~\ref{3133inS:fig} ---here and below
$(2h+1,3)$ \emph{graph} abbreviates \emph{graph with vertices
of valence} $(2h+1,3)$.
\begin{figure}
    \begin{center}
    \includegraphics[scale=.6]{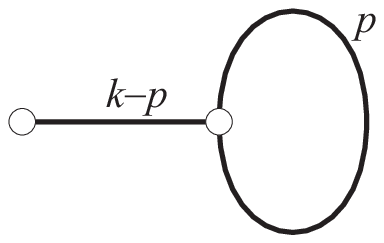}
    \hspace{1.5cm}
    \includegraphics[scale=.6]{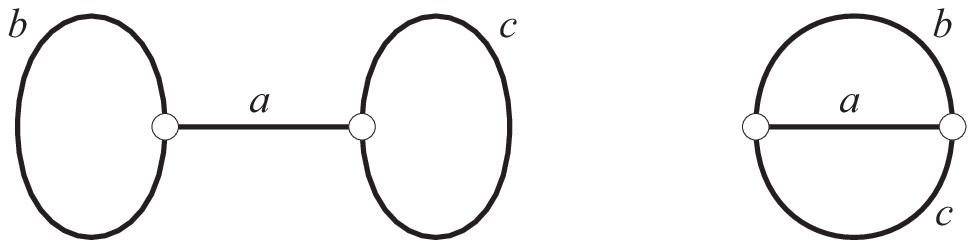}
    \end{center}
\mycap{The only $(1,3)$ graph and the two $(3,3)$ graphs in $S$. \label{3133inS:fig}}
\end{figure}
This graph gives a unique realization of $\pi=[p,2k-p]$ for $p<k$,
while $[k,k]$ is exceptional. Note that a single graph emerges for the
realization of the case with repeated partitions
$(S,S,4,3,[2,2],[1,3],[1,3])$, so its realization is \emph{a fortiori} unique
up to equivalence (and it is also immediate to check that the last move
generating $\sim$ leads this graph to itself).

\bigskip

Turning to the case $h=1$ we note that there are two $(3,3)$ graphs,
both with a unique embedding in $S$, shown  in Fig.~\ref{3133inS:fig}-center/right
and denoted by $\primo(a,b,c)$ and $\secon(a,b,c)$.
Remark that $\primo(a,b,c)$ has a $b\leftrightarrow c$ symmetry, while $\secon(a,b,c)$ is fully symmetric in $a,b,c$.
Moreover $\primo(a,b,c)$ realizes $\pi=[2a+b+c,b,c]$ while $\secon(a,b,c)$ realizes $[a+b,a+c,b+c]$.

Now we observe that the partition $\pi$ satisfies one and only one of the following:
\begin{itemize}
  \item[(i)] $\pi$ contains $k$;
  \item[(ii)] $\pi=[2k-2q,q,q]$ with (a) $1\leqslant q<\frac k2$ or (b) $\frac k2<q<k$;
  \item[(iii)] $\pi=[2k-q-r,q,r]$ with
\begin{itemize}
    \item[(a)] $1\leqslant r<\frac k2$ and $r<q<k-r$, or
    \item[(b)] $1\leqslant r<\frac k2$ and $k-r<q<k-\frac r2$, or
    \item[(c)] $\frac k2\leqslant r<\frac23k$ and $r<q<k-\frac r2$
\end{itemize}
\end{itemize}
(the conditions $2k-q-r>q>r$ readily imply that $r<\frac23k$ and $q<k-\frac r2$).
Since we always have $a+b+c=k$, it is immediate that neither $\primo(a,b,c)$ nor $\secon(a,b,c)$ can
realize (i). We now claim that for all the other listed cases there always is a single realization.
For $\pi=[2k-2q,q,q]$ there is a realization as $\primo(a,b,c)$ precisely if
$$\sistema{2a+b+c=2k-2q\\ b=c=q}\quad\Leftrightarrow\quad
\sistema{a=k-2q\\ b=c=q}\qquad \textrm{for}\ 1\leqslant q<\frac k2$$
so we have case (ii-a), while there is a realization as $\secon(a,b,c)$ if
$$\sistema{a+b=2k-2q\\ a+c=q\\ b+c=q}\quad\Leftrightarrow\quad
\sistema{a=k-q\\ b=k-q\\ c=2k-q}\qquad \textrm{for}\ \frac k2<q<k$$
whence case (ii-b). Turning to $[2k-q-r,q,r]$ with $2k-q-r>q>r$ there is a realization as $\primo(a,b,c)$ for
$$\sistema{2a+b+c=2k-q-r\\ b=q\\ c=r}\quad\Leftrightarrow\quad
\sistema{a=k-q-r\\ b=q\\ c=r}\quad
\begin{array}{l}
\textrm{for}\ r<q<k-r\\
\textrm{whence}\ r<\frac k2
\end{array}$$
and case (iii-a), while there is one as $\secon(a,b,c)$ if
$$\sistema{a+b=2k-q-r\\ a+c=q\\ b+c=r}\quad\Leftrightarrow\quad
\sistema{a=k-r\\ b=k-q\\ c=q+r-k}\quad
\begin{array}{l}
\textrm{for}\\
\max\{r,k-r\}<q<k
\end{array}$$
and depending on whether $\max\{r,k-r\}$ is $k-r$ or not we get
(iii-b) and (iii-c).
To conclude the case $h=1$, we must deal with the data
with repeated partitions, namely
$$(S,S,6,3,[2,2,2],[3,3],[2,2,2])$$
$$(S,S,8,3,[2,2,2,2],[3,3,2],[3,3,2]).$$
We begin with the former, noting that only one dessin d'enfant
$\Gamma$ arises from the previous argument for its realization,
so we must have $\nu=1$, as in the statement.  As a confirmation,
we check that of the $6$ potentially different graphs equivalent
to $\Gamma$ under $\sim$, for the other one $\Gamma'$
with black and white vertices of valence $[2,2,2]$ and $[3,3]$ respectively,
we actually have $\Gamma'=\gtil(\Gamma)$ for some $\gtil:S\to S$, as
apparent from Fig.~\ref{selfthetas:fig}-left.
\begin{figure}
    \begin{center}
    \includegraphics[scale=.6]{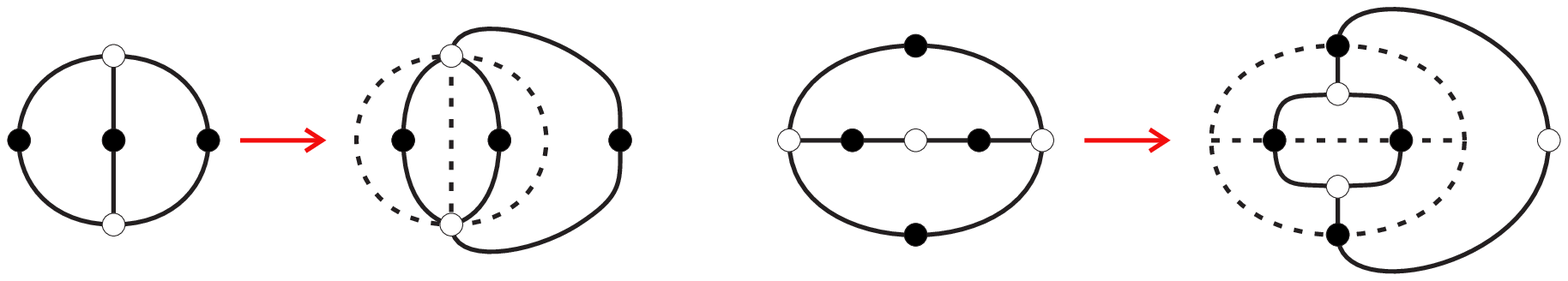}
    \end{center}
\mycap{Graphs equivalent under $\sim$. \label{selfthetas:fig}}
\end{figure}
For the latter datum with repetitions, the conclusion is analogous:
only one graph $\Gamma$ arises, so $\nu=1$, and if $\Gamma'$ is the
graph equivalent to $\Gamma$ under $\sim$ with black and white vertices
of valence $[2,2,2,2]$ and $[3,3,2]$ respectively, we have
have $\Gamma'=\gtil(\Gamma)$, as shown in Fig.~\ref{selfthetas:fig}-right.

\medskip

As an example, we provide in Table~\ref{Tk8:tab}
an application of Theorem~\ref{genus0:thm} for the case $h=1$ and $k=8$,
showing that all the listed cases actually occur.
\begin{table}
\begin{center}
\begin{tabular}{c||l|c|l}
$\pi$       &   Case
                &   $\nu$   &   Realizations \\ \hline\hline
(14,1,1)    & (ii-a) \tiny{$1\leqslant q=1<\frac k2=4$}
                & 1 & \small{\primo(6,1,1)}  \\ \hline
(13,2,1)    & (iii-a) \tiny{$\begin{array}{l}1\leqslant r=1<\frac k2=4\\ r=1<q=2<k-r=7\end{array}$}
                & 1 & \small{\primo(5,2,1)}  \\ \hline
(12,3,1)    & (iii-a) \tiny{$\begin{array}{l}1\leqslant r=1<\frac k2=4\\ r=1<q=3<k-r=7\end{array}$}
                & 1 & \small{\primo(4,3,1)}  \\ \hline
(12,2,2)    & (ii-a) \tiny{$1\leqslant q=2<\frac k2=4$}
                & 1 & \small{\primo(4,2,2)}  \\ \hline
(11,4,1)    & (iii-a) \tiny{$\begin{array}{l}1\leqslant r=1<\frac k2=4\\ r=1<q=4<k-r=7\end{array}$}
                & 1 & \small{\primo(3,4,1)}  \\ \hline
(11,3,2)    & (iii-a) \tiny{$\begin{array}{l}1\leqslant r=2<\frac k2=4\\ r=2<q=3<k-r=6\end{array}$}
                & 1 & \small{\primo(3,3,2)}  \\ \hline
(10,5,1)    & (iii-a) \tiny{$\begin{array}{l}1\leqslant r=1<\frac k2=4\\ r=1<q=5<k-r=7\end{array}$}
                & 1 & \small{\primo(2,5,1)}  \\ \hline
(10,4,2)    & (iii-a) \tiny{$\begin{array}{l}1\leqslant r=2<\frac k2=4\\ r=2<q=4<k-r=6\end{array}$}
                & 1 & \small{\primo(2,4,2)}  \\ \hline
(10,3,3)    & (ii-a) \tiny{$1\leqslant q=3<\frac k2=4$}
                & 1 & \small{\primo(2,3,3)}  \\ \hline
(9,6,1)     & (iii-a) \tiny{$\begin{array}{l}1\leqslant r=1<\frac k2=4\\ r=1<q=6<k-r=7\end{array}$}
                & 1 & \small{\primo(1,6,1)}  \\ \hline
(9,5,2)     & (iii-a) \tiny{$\begin{array}{l}1\leqslant r=2<\frac k2=4\\ r=2<q=5<k-r=6\end{array}$}
                & 1 & \small{\primo(1,5,2)}  \\ \hline
(9,4,3)     & (iii-a) \tiny{$\begin{array}{l}1\leqslant r=3<\frac k2=4\\ r=3<q=4<k-r=5\end{array}$}
                & 1 & \small{\primo(1,4,3)}  \\ \hline
(8,7,1)     & (i)
                & 0 &   \\ \hline
(8,6,2)     & (i)
                & 0 &   \\ \hline
(8,5,3)     & (i)
                & 0 &   \\ \hline
(8,4,4)     & (i)
                & 0 &   \\ \hline
(7,7,2)     & (ii-b) \tiny{$\frac k2=4<q=7<k=8$}
                & 1 & \small{\secon(6,1,1)}  \\ \hline
(7,6,3)     & (iii-b) \tiny{$\begin{array}{l}1\leqslant r=3<\frac k2=4\\ k-r=5<q=6<k-\frac r2=6.5\end{array}$}
                & 1 & \small{\secon(5,2,1)}  \\ \hline
(7,5,4)     & (iii-c) \tiny{$\begin{array}{l}\frac k2=4\leqslant r=4<\frac23k=5.\overline{3}\\ r=4<q=5<k-\frac r2=6\end{array}$}
                & 1 & \small{\secon(4,3,1)}  \\ \hline
(6,6,4)     & (ii-b) \tiny{$\frac k2=4<q=6<k=8$}
                & 1 & \small{\secon(4,2,2)}  \\ \hline
(6,5,5)     & (ii-b) \tiny{$\frac k2=4<q=5<k=8$}
                & 1 & \small{\secon(3,3,2)}
\end{tabular}
\end{center}
\mycap{The genus-$0$ case with $h=1$ and $k=8$. \label{Tk8:tab}}
\end{table}

\bigskip

Now we concentrate on the case $h=2$, that requires considerable work.
We first note that there are two abstract $(5,3)$ graphs, one of which
has two inequivalent embeddings in $S$, as shown in Fig.~\ref{53inS:fig}.
\begin{figure}
    \begin{center}
    \includegraphics[scale=.6]{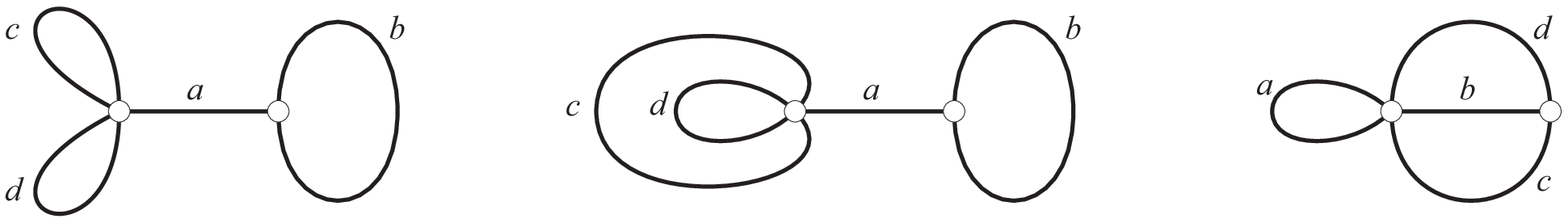}
    \end{center}
\mycap{The $(5,3)$ graphs in $S$. \label{53inS:fig}}
\end{figure}
Denoting these graphs by $\primo(a,b,c,d),\ \secon(a,b,c,d),\ \terzo(a,b,c,d)$, we note that
they realize the partitions $[2a+b+c+d,b,c,d]$, $[2a+b+c,c+d,b,d]$, $[a+c+d,b+c,b+d,a]$ respectively,
and that their only symmetries are $\primo(a,b,d,c)=\primo(a,b,c,d)$ and $\terzo(a,b,d,c)=\terzo(a,b,c,d)$.

The form of the partitions realized by $\primo,\secon,\terzo$ readily
shows that $\pi=[p,p,p,p]$ cannot be realized. Moreover, $\pi=[p,p,q,q]$
for $p>q$ of course cannot via $\primo$, and it also cannot via
$\secon$ or $\terzo$, since it only could as
$$\sistema{2a+b+c=p\\ c+d=p\\ b=k-p\\ d=k-p}
\quad\Rightarrow a=0,\qquad\qquad
\sistema{a+c+d=p\\ b+c=p\\ b+d=k-p\\ a=k-p}
\quad\Rightarrow d=0.$$

\medskip

The partition $\pi=[p,p,p,q]$ for $p>q$, so $q=2k-3p$ and $\frac k2<p<\frac23k$,
cannot be realized via $\primo$ or $\secon$, while it can
via $\terzo$ only as
$$\sistema{a+c+d=p\\ b+c=p\\ b+d=p\\ a=2k-3p}
\quad\Leftrightarrow\quad\sistema{a=2k-3p\\ b=k-p\\ c=2p-k\\ d=2p-k}$$
which gives positive $a,b,c,d$, so there is a unique realization.
Note that $\pi$ cannot contain $k$, so our finding is in agreement with the statement.

\medskip

Turning to $\pi=[p,q,q,q]$ for $p>q$, so $p=2k-3q$ and $0<q<\frac k2$, we get
realizations via $\primo$ as
$$\sistema{2a+b+c+d=2k-3q\\ b=c=d=q}
\quad\Leftrightarrow\quad\sistema{a=k-3q\\ b=c=d=q}\qquad\textrm{for}\ q<\frac k3,$$
while there is none via $\secon$, and there is one via $\terzo$ as
$$\sistema{a+c+d=2k-3q\\ b+c=b+d=a=q}
\quad\Leftrightarrow\quad\sistema{a=q\\ b=3q-k\\ c=k-2q\\ d=k-2q}\qquad\textrm{for}\ q>\frac k3$$
so there is a unique realization except for $q=\frac k3$, namely $p=k$, as in the statement.

\medskip

Before proceeding we prove two facts that we will use a few times.
Take $\pi=[p,q,r,s]$ with $p\geqslant q\geqslant r\geqslant s$.
Then:
\begin{itemize}
\item If $\pi$ contains $k$ then $p=k$;
\item If $k$ is the sum of two entries of $\pi$ then $p+s=q+r=k$.
\end{itemize}
The first assertion is obvious. For the second one, note that
if $p+r=q+s=k$ then $p=q$ and $r=s$, so we also have $p+s=q+r=k$,
while if $p+q=r+s=k$ then $p=q=r=s$, and again $p+s=q+r=k$.

\medskip

We now study the partitions $\pi=[p,p,q,r]$ with $p>q>r$.
Note that $r=2k-2p-q$, so $q<2k-2p$, and $q>2k-2p-q$, so $q>k-p$,
and finally $2p>q+r$, so $p>\frac k2$.
No realization via $\primo$ is possible, while one via $\secon$ would be only in
one of the following ways:
$$\sistema{2a+b+c=p\\ c+d=p\\ b=q\\ d=2k-2p-q}
\quad\Rightarrow\quad a=k-p-q\quad\Rightarrow\quad q<k-p\ \textrm{(impossible)},$$
$$\sistema{2a+b+c=p\\ c+d=p\\ b=2k-2p-q\\ d=q}
\quad\Leftrightarrow\quad \sistema{a=p+q-k\\ b=2k-2p-q\\ c=p-q\\ d=q.}$$
The first way gives nothing, and the second one gives a unique solution without conditions.
Similarly, via $\terzo$ we could only have
$$\sistema{a+c+d=p\\ b+c=p\\ b+d=q\\ a=2k-2p-q}
\quad\Leftrightarrow\quad\sistema{a=2k-2p-q\\ b=k-p\\ c=2p-k\\ d=p+q-k,}$$
$$\sistema{a+c+d=p\\ b+c=p\\ a=q\\ b+d=2k-2p-q}
\quad\Rightarrow\quad d=k-p-q \quad\Rightarrow\quad q<k-p\ \textrm{(impossible)},$$
$$\sistema{b+c=p\\ b+d=p\\ a+c+d=q\\ a=2k-2p-q}
\quad\Leftrightarrow\quad\sistema{a=2k-2p-q\\ b=k-q\\ c=p+q-k\\ d=p+q-k}$$
where the second way gives nothing and the other two always give acceptable
solutions, so we have $2$ realizations via $\terzo$. The total number of
realizations of $(p,p,q,r)$ is then always $3$. To show that this is
in agreement with the statement, we note that we cannot have $p=k$, and we also
cannot have $p+r=p+q=k$, otherwise $r=q$.

\medskip

Turning to $(p,q,q,r)$ for $p>q>r$, we first note that
$r=2k-p-2q$, so $q<k-\frac p2$, and $q>2k-p-2q$, so $q>\frac13(2k-p)$.
Moreover $p>\frac k2$, whence $3k-3p<2k-p$ and $k-p<\frac13(2k-p)$,
therefore $q>k-p$, which we will need below.
The realizations via $\primo$ come from
$$\sistema{2a+b+c+d=p\\ c=q\\ d=q\\ b=2k-p-2q}
\quad\Leftrightarrow\quad\sistema{a=p-k\\ b=2k-p-2q\\ c=q\\ d=q}\qquad\textrm{for}\ p>k,$$
$$\sistema{2a+b+c+d=p\\ c=q\\ d=2k-p-2q\\ b=q}
\quad\Leftrightarrow\quad\sistema{a=p-k\\ b=q\\ c=q\\ d=2k-p-2q}\qquad\textrm{for}\ p>k$$
whence two solutions for $p>k$ and none otherwise. Via $\secon$ we can have
$$\sistema{2a+b+c=p\\ b=q\\ c+d=q\\ d=2k-p-2q}
\quad\Leftrightarrow\quad\sistema{a=k-2q\\ b=q\\ c=p+3q-2k\\ d=2k-p-2q}
\qquad \textrm{for}\ q<\frac k2,$$
$$\sistema{c+d=p\\ 2a+b+c=q\\ d=q\\ b=2k-p-2q}
\quad\Leftrightarrow\quad\sistema{a=2q-k\\ b=2k-p-2q\\ c=p-q\\ d=q}
\qquad\begin{array}{l}
\textrm{for}\ q>\frac k2\\ (\textrm{so}\ p<k),
\end{array}$$
therefore, no solution for $q=\frac k2$ and one otherwise.
Finally, from $\terzo$ we get
$$\sistema{a+c+d=p\\ b+c=q\\ b+d=q\\ a=2k-p-2q}
\quad\Leftrightarrow\quad\sistema{a=2k-p-2q\\ b=k-p\\ c=p+q-k\\ d=p+q-k}\qquad\textrm{for}\ p<k$$
$$\sistema{a+c+d=p\\ b+c=q\\ b+d=2k-p-2q\\ a=q}
\quad\Leftrightarrow\quad\sistema{a=q\\ b=k-p\\ c=p+q-k\\ d=k-2q}\qquad\textrm{for}\ p<k\ \textrm{and}\ q<\frac k2$$
$$\sistema{a+c+d=q\\ b+c=p\\ b+d=q\\ a=2k-p-2q}
\quad\Leftrightarrow\quad\sistema{a=2k-p-2q\\ b=k-q\\ c=p+q-k\\ d=2q-k}\qquad
\begin{array}{l}
\textrm{for}\ q>\frac k2\\ (\textrm{so}\ p<k).
\end{array}$$
Summing up, for $(p,q,q,r)$ we have $1$ realization if $p=k$ or $q=\frac k2$, and $3$
otherwise, in accordance with the statement.

\medskip

In the last case with repetitions, namely $\pi=[p,q,r,r]$ with $p>q>r$, we note that $p-q$ is even,
and we denote it by $2j$, so $q=p-2j$ and $r=k-p+j$, with $p-2j>k-p+j$, so
$p-k<j<\frac13(2p-k)$. We also note that $p>\frac k2$, which readily implies that
$\frac13(2p-k)<p-\frac k2$, therefore $j<p-\frac k2$, which we will need soon.
The realizations of $\pi$ via $\primo$ come as
$$\sistema{2a+b+c+d=p\\ b=p-2j\\ c=d=k-p+j}
\quad\Leftrightarrow\quad\sistema{a=p-k\\ b=p-2j\\ c=d=k-p+j}\qquad\textrm{for}\ p>k,$$
$$\sistema{2a+b+c+d=p\\ b=d=k-p+j\\ c=p-2j}
\quad\Leftrightarrow\quad\sistema{a=p-k\\ b=d=k-p+j\\ c=p-2j}\qquad\textrm{for}\ p>k$$
so there are $2$ of them for $p>k$ and none otherwise. Now $\secon$ gives
$$\sistema{2a+b+c=p\\ c+d=p-2j\\ b=d=k-p+j}
\quad\Leftrightarrow\quad\sistema{a=j\\ b=d=k-p+j\\ c=2p-k-3j,}$$
$$\sistema{2a+b+c=p-2j\\ c+d=p\\ b=d=k-p+j}
\quad\Rightarrow\quad a=-j\ \textrm{(impossible)}$$
whence always a unique realization. From $\terzo$ we get instead
$$\sistema{a+c+d=p\\ b+c=p-2j\\ b+d=a=k-p+j}
\quad\Leftrightarrow\quad\sistema{a=k-p+j\\ b=k-p\\ c=2p-2j-k\\ d=j}\qquad\textrm{for}\ p<k,$$
$$\sistema{a+c+d=p\\ a=p-2j\\ b+c=b+d=k-p+j}
\quad\Leftrightarrow\quad\sistema{a=p-2j\\ b=k-p\\ c=d=j}\qquad\textrm{for}\ p<k,$$
$$\sistema{a+c+d=p-2j\\ a=b+d=k-p+j\\ b+c=p}
\quad\Rightarrow\quad d=-j$$
whence $2$ realizations for $p<k$ and none otherwise.
Summarizing the case $\pi=[p,q,r,r]$, we have one realization
for $p=k$ and $3$ otherwise, which agrees with the statement, since
for such a $\pi$ the conditions $p+r=q+r=k$ implies $p=q$.

\medskip

We are only left to deal with the general case
$\pi=[p,q,r,s]$ for $p>q>r>s$, so $s=2k-p-q-r$
with $0<2k-p-q-r<r$, so $k-\frac12(p+q)<r<2k-p-q$.
Note that we also have $p+q>p+r>k$ and
$r+s<q+s<k$.
From $\primo$ we get
$$\sistema{2a+b+c+d=p\\ b=q\\ c=r\\ d=2k-p-q-r}
\quad\Leftrightarrow\quad\sistema{a=p-k\\ b=q\\ c=r\\ d=2k-p-q-r}\qquad\textrm{for}\ p>k$$
plus two more instances with $b=r,\ \{c,d\}=\{q,s\}$ and $b=s,\ \{c,d\}=\{q,r\}$,
always with an acceptable solution with $a=k-p$, so \primo\
gives $3$ realizations of $\pi$ if $p>k$ and none otherwise.
From $\secon$ we get instead
$$\sistema{2a+b+c=p\\ b=q\\ c+d=r\\ d=2k-p-q-r}
\quad\Leftrightarrow\quad\sistema{a=k-q-r\\ b=q\\ c=p+q+2r-2k\\ d=2k-p-q-r}\qquad\textrm{for}\ q+r<k$$
$$\sistema{2a+b+c=p\\ c+d=q\\ b=r\\ d=2k-p-q-r}
\quad\Leftrightarrow\quad\sistema{a=k-q-r\\ b=r\\ c=p+2q+r-2k\\ d=2k-p-q-r}\qquad\textrm{for}\ q+r<k,$$
(note that $p+2q+r-2k>p+q+2r-2k>0$),
$$\sistema{2a+b+c=p\\ c+d=q\\ d=r\\ b=2k-p-q-r}
\quad\Leftrightarrow\quad\sistema{a=p+r-k\\ b=2k-p-q-r\\ c=q-r\\ d=r,}$$
$$\sistema{c+d=p\\ 2a+b+c=q\\ b=r\\ d=2k-p-q-r}
\quad\Rightarrow\quad a=k-p-r\quad\Rightarrow\quad 
\begin{array}{l}
p+r<k\\ 
\textrm{(impossible),}
\end{array}
$$
$$\sistema{c+d=p\\ 2a+b+c=q\\ d=r\\ b=2k-p-q-r}
\quad\Leftrightarrow\quad\sistema{a=q+r-k\\ b=2k-p-q-r\\ c=p-r\\ d=r}\qquad\textrm{for}\ q+r>k,$$
$$\sistema{c+d=p\\ d=q\\ 2a+b+c=r\\ b=2k-p-q-r}
\quad\Leftrightarrow\quad\sistema{a=q+r-k\\ b=2k-p-q-r\\ c=p-q\\ d=q}\qquad\textrm{for}\ q+r>k$$
whence $1$ realization of $\pi$ for $q+r=k$ and $3$ otherwise.
Finally, from $\terzo$ we get
$$\sistema{a+c+d=p\\ a=q\\ b+c=r\\ b+d=2k-p-q-r}
\quad\Leftrightarrow\quad\sistema{a=q\\ b=k-p\\ c=p+r-k\\ d=k-q-r}\qquad\begin{array}{l}\textrm{for}\ p<k\\ \textrm{and}\ q+r<k,\end{array}$$
$$\sistema{a+c+d=p\\ b+c=q\\ a=r\\ b+d=2k-p-q-r}
\quad\Leftrightarrow\quad\sistema{a=r\\ b=k-p\\ c=p+q-k\\ d=k-q-r}\qquad\begin{array}{l}\textrm{for}\ p<k\\ \textrm{and}\ q+r<k,\end{array}$$
$$\sistema{a+c+d=p\\ b+c=q\\ b+d=r\\ a=2k-p-q-r}
\quad\Leftrightarrow\quad\sistema{a=2k-p-q-r\\ b=k-p\\ c=p+q-k\\ d=p+r-k}\qquad\textrm{for}\ p<k,$$
$$\sistema{b+c=p\\ a+c+d=q\\ a=r\\ b+d=2k-p-q-r}
\quad\Rightarrow\quad d=k-p-r\quad\Rightarrow\quad 
\begin{array}{l}
p+r<k\\
\textrm{(impossible),}
\end{array}$$
$$\sistema{b+c=p\\ a+c+d=q\\ b+d=r\\ a=2k-p-q-r}
\quad\Leftrightarrow\quad\sistema{a=2k-p-q-r\\ b=k-q\\ c=p+q-k\\ d=q+r-k}\qquad\textrm{for}\ q+r>k,$$
$$\sistema{b+c=p\\ b+d=q\\ a+c+d=r\\ a=2k-p-q-r}
\quad\Leftrightarrow\quad\sistema{a=2k-p-q-r\\ b=k-r\\ c=p+r-k\\ d=q+r-k}\qquad\textrm{for}\ q+r>k.$$
Since $q+r>k$ implies $p<k$, we conclude that \terzo\ yields $3$ realizations of $\pi$
for $p<k$ and $q+r\neq k$, only $1$ if $p<k$ and $q+r=k$, and none if $p\geqslant k$.
Noting that the conditions $q+r=k$ and $p=k$ are mutually exclusive, we
conclude that we have the number of realizations of $\pi$ is as follows:
$2$ if $q+r=k$ ($1$ from $\secon$ and $1$ from $\terzo$), then $3$ if $p=k$
(from $\secon$), and $6$ otherwise (always $3$ from $\secon$, plus $3$ from
$\primo$ is $p>k$ and $3$ from $\terzo$ if $p<k$. This is precisely what the statement says.

\medskip

To conclude the case $h=2$ we must take into account the
datum with repeated partitions $(S,S,12,3,[2,2,2,2,2,2],[5,3,2,2],[5,3,2,2])$,
for which the statement gives $\nu=3$, coming from the realizations
of the datum via the graphs
$\secon(1,2,1,2),\ \terzo(3,1,1,1),\ \terzo(2,1,2,1)$
according to the above argument.
So to verify
that our computation of $\nu=3$ is correct also in this case we must show that
these graphs are pairwise inequivalent under $\sim$, which amounts
to showing that by applying to each of them the last move
generating $\sim$ we always get the same graph again, and not one of the other two.
This is done in Fig.~\ref{53g0selfdual:fig}.
\begin{figure}
    \begin{center}
    \includegraphics[scale=.6]{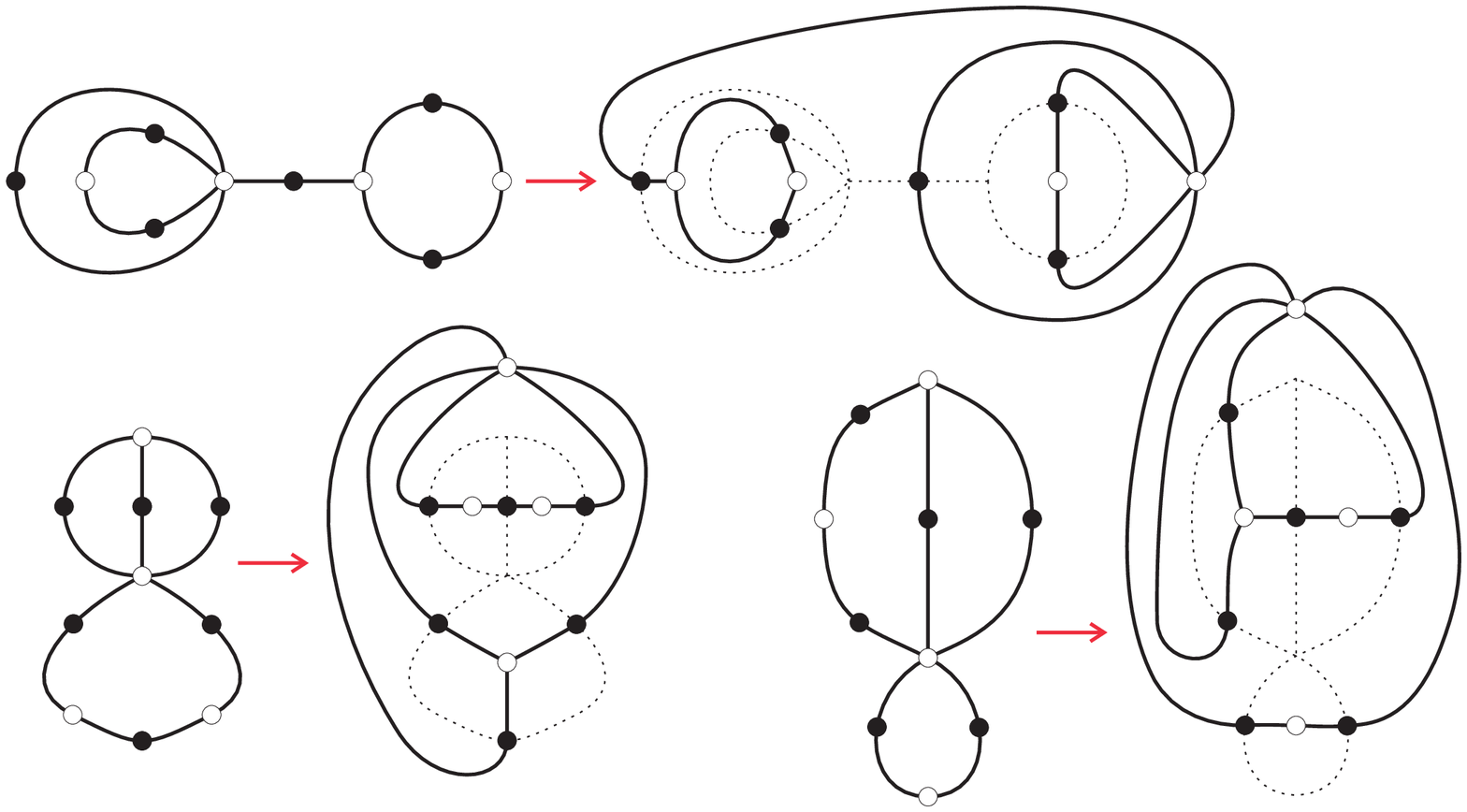}
    \end{center}
\mycap{Graphs mapped to themselves by the move generating $\sim$. \label{53g0selfdual:fig}}
\end{figure}

\section{Genus 1}\label{genus1:sec}
In this section we prove Theorem~\ref{genus1:thm}, starting
from the case $h=1$. We have to consider the embeddings in $T$ of the
$(3,3)$ graphs of Fig.~\ref{3133inS:fig} with a single disc as a region.
For the first graph, of course neither of the closed edges can be trivial
in $T$, so the situation must be as in Fig.~\ref{33inT:fig}-left,
\begin{figure}
    \begin{center}
    \includegraphics[scale=.6]{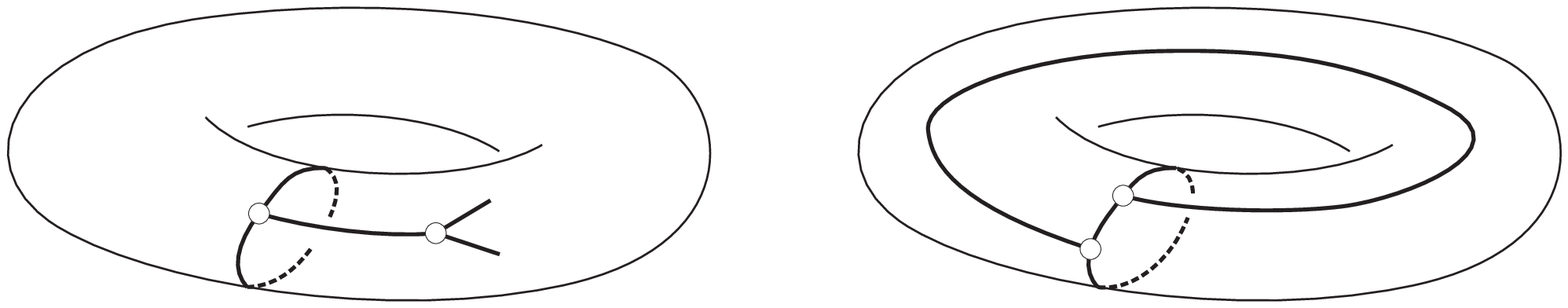}
    \end{center}
\mycap{The only $(3,3)$ graph that embeds in $T$ with a single disc as region. \label{33inT:fig}}
\end{figure}
but then the complement has disconnected boundary. One also easily sees that the second
graphs embeds only as shown in Fig.~\ref{33inT:fig}-right and with a full
$\permu_3\times\matZ/_2$ symmetry, so $\nu$ is the number of ways to write $k$ as
the sum of $3$ unordered positive integers, which is $\binom{k-1}{2}$ as claimed.

\bigskip

Let us now consider the case $h=2$, whence $\ell=2$. We first recall that
there are two abstract $(5,3)$ graphs, already shown in Fig.~\ref{53inS:fig}
(we can ignore here the graph in the centre, that
represents a different embedding in $S$ of that on the left).
We now claim that they have exactly the inequivalent embeddings in $T$
shown in Fig.~\ref{53inT:fig}
\begin{figure}
    \begin{center}
    \includegraphics[scale=.6]{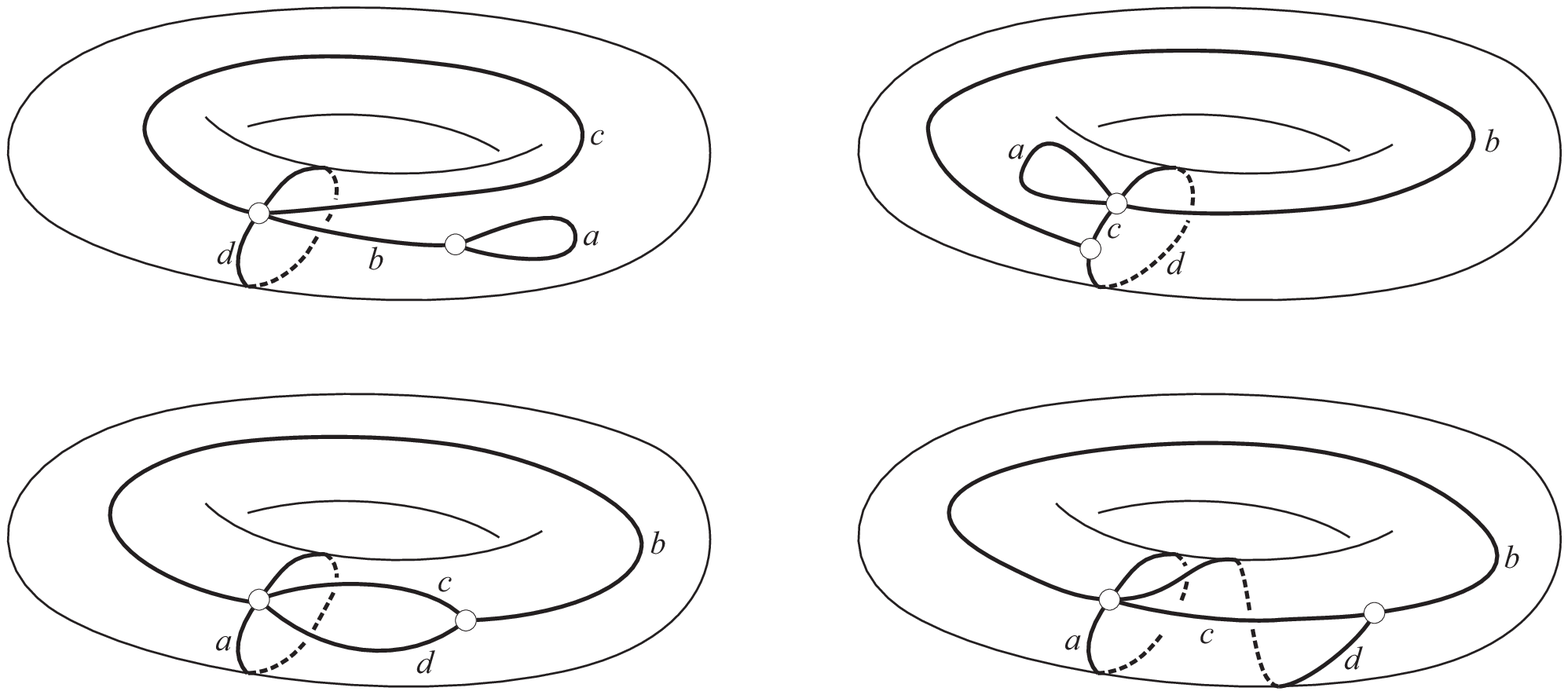}
    \end{center}
\mycap{The $(5,3)$ graphs in $T$ with two regions. \label{53inT:fig}}
\end{figure}
and denoted by $\primo(a,b,c,d)$, $\secon(a,b,c,d)$, $\terzo(a,b,c,d)$ and $\quart(a,b,c,d)$.
To this end, we first note that the graph of Fig.~\ref{53inS:fig}-left cannot embed with
the loop $b$ being non-trivial in $T$, otherwise one of the regions would
not be a disc.  Knowing that $b$ is trivial we easily get the only embedding
$\primo(a,b,c,d)$. Turning to the graph of Fig.~\ref{53inS:fig}-right, we ask ourselves
whether the loop $a$ is trivial in $T$ or not. If it is, we easily get $\secon(a,b,c,d)$.
If $a$ is non-trivial, the edges $a,b,c$ cannot leave the vertex of $a$ from the same side,
so assume $b$ leaves from one side and $c,d$ from the opposite side. It is then easy to see
that the only possibilities are $\terzo(a,b,c,d)$ and $\quart(a,b,c,d)$.

We next claim that all four graphs have the only symmetry $c\leftrightarrow d$.
Checking that there cannot be any other one is immediate, while we prove
that $c\leftrightarrow d$ exists for $\quart(a,b,c,d)$, which is the hardest case
to visualize. Note that, with suitable orientations, the attaching maps of the
complementary discs to $\quart(a,b,c,d)$ are described by the words
$bc^{-1}db^{-1}a^{-1}$ and $acd^{-1}$. Now we consider the automorphism of
$\quart(a,b,c,d)$ that maps $a,b,c,d$ to $a^{-1},b,d,c$. This transforms the
attaching words into $bd^{-1}cb^{-1}a$ and $a^{-1}dc^{-1}$, whose inverses are
$a^{-1}bc^{-1}db^{-1}$ and $cd^{-1}a$, which are cyclically identical to
the initial ones. So the automorphism of $\quart(a,b,c,d)$ extends to $T$,
and we are done.

To count the realizations of a partition $(p,2k-p)$ with $p\leqslant k$, we note that
$\primo(a,b,c,d)$ and $\secon(a,b,c,d)$
both realize $[a+2b+2c+2d,a]$, while $\terzo(a,b,c,d)$
realizes $[2a+2b+c+d,c+d]$, and $\quart(a,b,c,d)$ realizes $[a+2b+c+d,a+c+d]$.
The conclusion is now easy. Of course $\nu=0$ if $p=k$, while for $p<k$
in $\primo$ and $\secon$ we must have $a=p$, so they both contribute with
$$\sum_{b=1}^{k-p-2}\left[\frac{k-p-b}2\right]=\left[\frac14(k-p-1)^2\right];$$
in $\terzo$ we must have $c+d=p$ and $a+b=k-p$, whence
$$\left[\frac p2\right]\cdot (k-p-1);$$
finally, in $\quart$ we must have $a+c+d=p$ and $b=k-p$, whence
$$\sum_{a=1}^{p-2}\left[\frac{p-a}2\right]=\left[\frac14(p-1)^2\right].$$
To conclude the case $h=2$ we need to consider the datum with repetitions $$(T,S,8,3,[2,2,2,2],[5,3],[5,3]),$$
but the previous discussion implies that only the graph $\quart(1,1,1,1)$ can realize it,
so the value $\nu=1$ already obtained is correct.  In particular, the graph must be mapped to itself
by the last move generating $\sim$, which one can verify rather easily.

\section{Genus 2}\label{genus2:sec}
We start by showing in Fig.~\ref{73abstract:fig}
\begin{figure}
    \begin{center}
    \includegraphics[scale=.6]{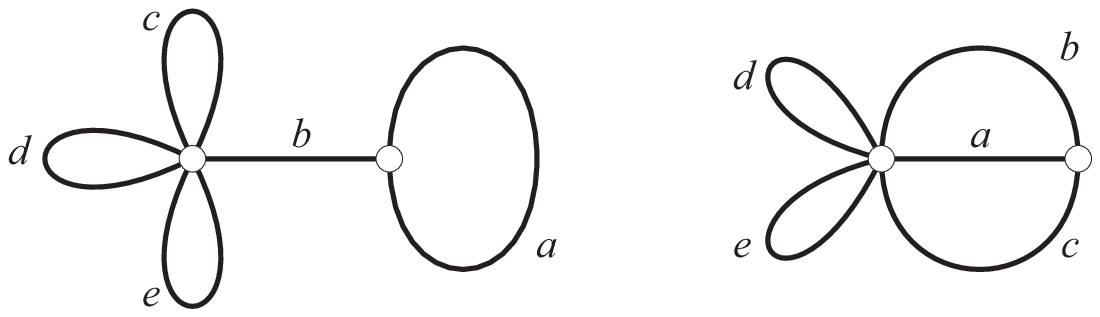}
    \end{center}
\mycap{The abstract $(7,3)$ graphs. \label{73abstract:fig}}
\end{figure}
the only two abstract $(7,3)$ graphs, $\Gamma_1$ and $\Gamma_2$.
Next, we note that $\Gamma_1$ cannot embed in the genus-2 surface $2T$ with
a single disc discal region: if $a$ bounds in $2T$ a disc disjoint from the
rest of $\Gamma_1$, then there is more than one region, otherwise
a regions is non-discal.

We must then enumerate up to symmetry
the embeddings of $\Gamma_2$ in $2T$ with a single region.
We do so by describing the fattenings of $\Gamma_2$ to a ribbon having
a single boundary component attaching a disc to which we get $2T$.
Note that a fattening of a graph can be described by an immersion in the
plane. There are only two fattenings of the subgraph
$\theta=a\cup b\cup c$ of $\Gamma_2$, shown in Fig.~\ref{planarthetas:fig}
\begin{figure}
    \begin{center}
    \includegraphics[scale=.6]{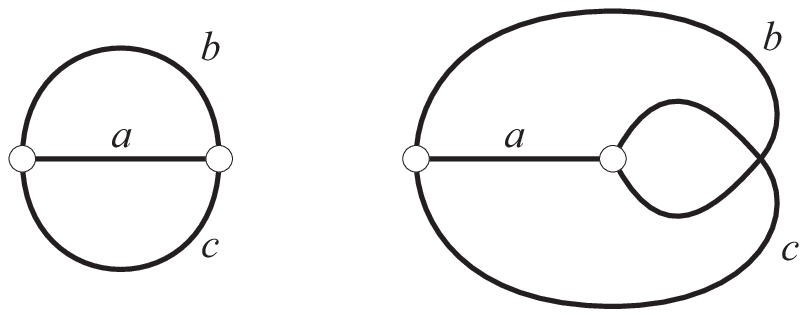}
    \end{center}
\mycap{Fattenings of the graph $\theta$. \label{planarthetas:fig}}
\end{figure}
and both seen to be totally symmetric. We now concentrate on the fattenings of
the subgraph $\Delta$ of $\Gamma_2$ given by a regular neighbourhood
of the $7$-valent vertex union $d\cup e$. One easily sees that there
are $12$ fattenings of $\Delta$, but $7$ of them cannot give rise
to fattenings of $\Gamma_2$ with a connected boundary, because a small
boundary circle is already created near $\Delta$.  The other $5$
fattenings of $\Delta$ are shown in Fig.~\ref{fatDelta:fig}.
\begin{figure}
    \begin{center}
    \includegraphics[scale=.6]{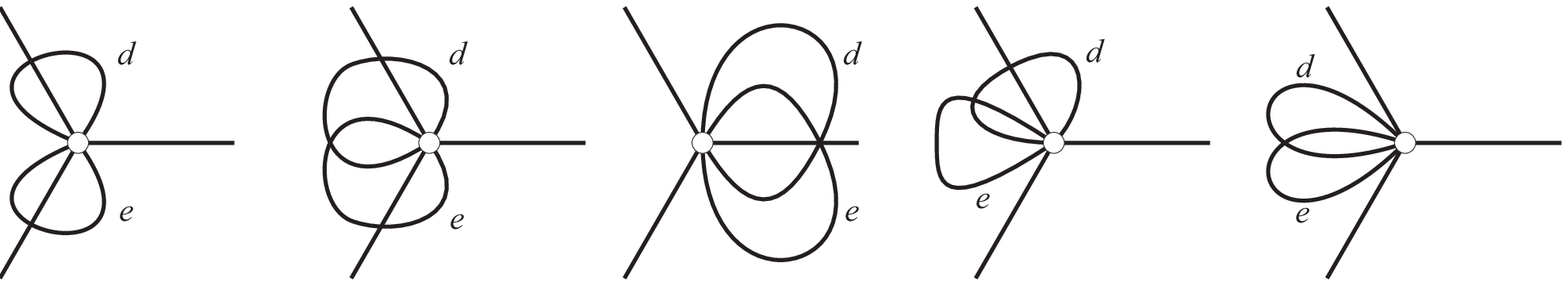}
    \end{center}
\mycap{Fattenings of $\Delta$ without small boundary circles. \label{fatDelta:fig}}
\end{figure}
We must now combine the $2$ fattenings of $\theta$
with the $5$ of $\Delta$, and thanks to the stated symmetry of those of $\theta$ there is
only one way to combine any given pair of fattenings. Of the resulting $10$ ribbons,
$4$ turn out to have disconnected boundary, and the other $6$ are shown in Fig.~\ref{2Tribbons:fig}.
\begin{figure}
    \begin{center}
    \includegraphics[scale=.6]{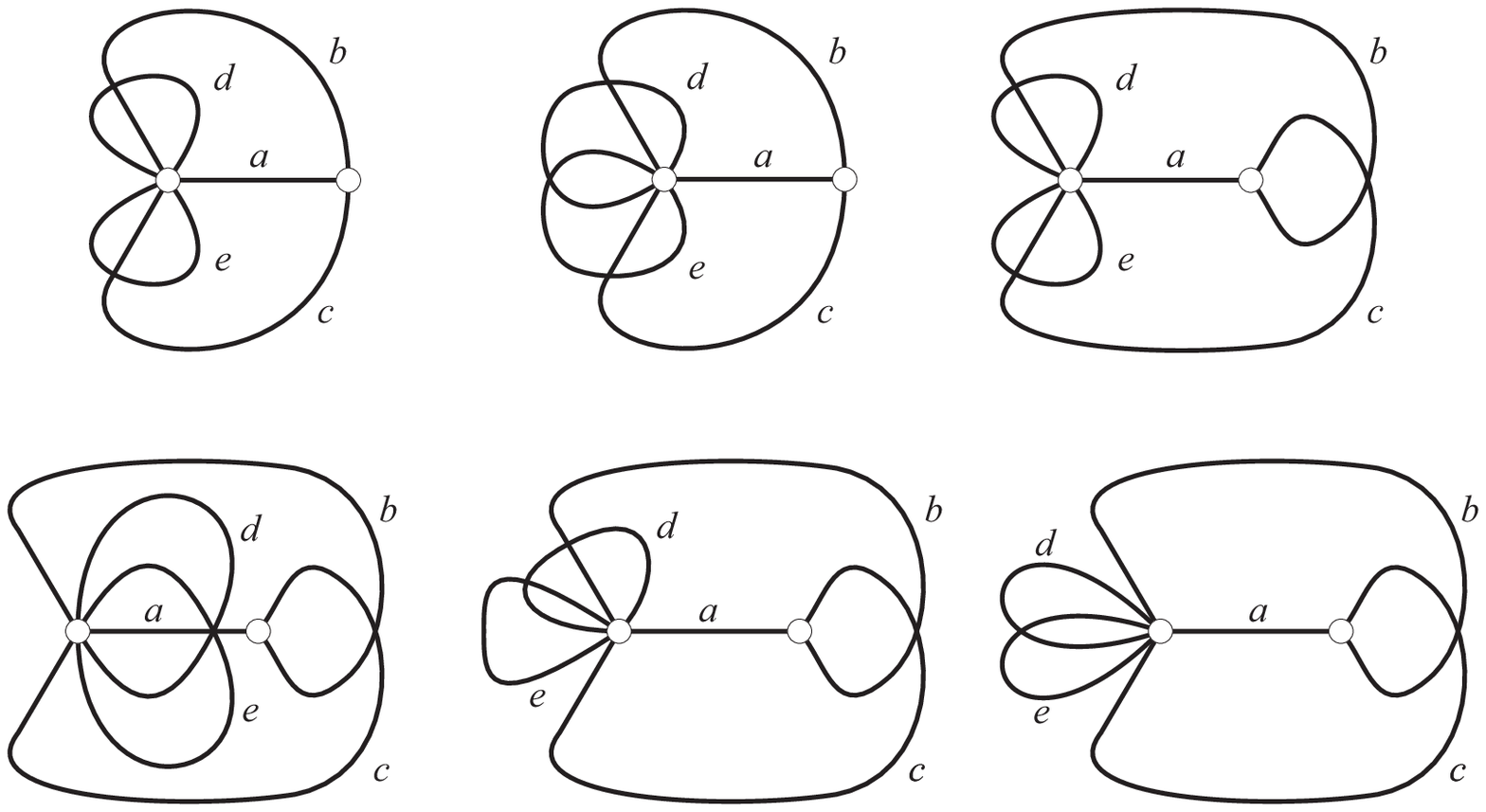}
    \end{center}
\mycap{Fattenings of $\Gamma_2$ with a single boundary circle, capping off which we get $2T$. \label{2Tribbons:fig}}
\end{figure}

Since $5$ of the graphs of Fig.~\ref{2Tribbons:fig}
have a symmetry  of type $(b\leftrightarrow c,d\leftrightarrow e)$ and one has no symmetries,
the number $\nu$ of realizations of the datum
is given by $5$ times the number $x$ of ways to express $k$ as $a+b+c+d+e$ up to $(b\leftrightarrow c,d\leftrightarrow e)$,
plus the number $y$ of ways to express $k$ as $a+b+c+d+e$ with no symmetries to take into account.
Of course $y=\binom{k-1}4$, while
$$x=\sum_{a=1}^{k-4}z(k-a)=\sum_{h=4}^{k-1}z(h),$$
where $z(h)$ is the number of
ways to express $h$ as $b+c+d+e$ up to $(b\leftrightarrow c,d\leftrightarrow e)$.

We can now compute $z(h)$ by distinguishing the case $b<c$ from $b=c,\ d<e$ and from $b=c,\ d=e$.
For the case $b<c$ we can choose $j=b+c+d$ between $4$ and $h-1$, then $i=b+c$ between $3$ and $j-1$,
and we are left with $\left[\frac{i-1}2\right]$ choices for $b$ and $c$, so we have a contribution to $z(h)$ equal to
$$\sum_{j=4}^{h-1}\sum_{i=3}^{j-1}\left[\frac{i-1}2\right]=
\sum_{j=4}^{h-1}\sum_{i=2}^{j-2}\left[\frac i2\right]=
\sum_{j=4}^{h-1}\left[\left(\frac{j-2}2\right)^2\right]=
\sum_{j=4}^{h-1}\left[\left(\frac j2-1\right)^2\right].$$
We can now distinguish between the odd case $h=2m+1$ and the even case $h=2m$,
splitting the sum between the odd and the even values of $j$, and getting respectively, after easy calculations,
$$\frac16m(m-1)(4m-5)\qquad\textrm{and}\qquad\frac16(m-1)(m-2)(4m-3).$$

For the case $b=c,\ d<e$ we can choose $b=c$ between $1$ and $\left[\frac{h-3}2\right]$, so
we have a contribution to $z(h)$ equal to
\begin{eqnarray*}
 & & \sum_{b=1}^{\left[\frac{h-3}2\right]}\left[\frac{h-2b-1}2\right]=
  \sum_{b=1}^{\left[\frac{h-3}2\right]}\left(\left[\frac{h-1}2\right]-b\right) \\
 &=& \left[\frac{h-1}2\right]\cdot \left[\frac{h-3}2\right]-\frac12\left[\frac{h-3}2\right]\left(\left[\frac{h-3}2\right]+1\right) \\
 &=& \frac12\left[\frac{h-3}2\right]\left(2\left(\left[\frac{h+3}2\right]+1\right)-
    \left(\left[\frac{h-3}2\right]+1\right)\right) \\
 &=&\frac12\left[\frac{h-3}2\right]\left(\left[\frac{h-3}2\right]+1\right).
\end{eqnarray*}
Therefore we have a contribution to $z(h)$ for $h=2m+1$ and for $h=2m$ given respectively by
$$\frac12m(m-1)\qquad\textrm{and}\qquad\frac12(m-2)(m-1).$$

Finally, for $b=c,\ d=e$ we have a contribution of $0$ for odd $h$ and of
$m-1$ for $h=2m$.

We can now plug these contributions to $z(h)$ in the formula for $x$. Again we distinguish between the case
$k=2p+1$ and the case $k=2p$, splitting the sum between the odd $h=2m+1$ and the even $h=2m$, getting
respectively
\begin{eqnarray*}
& &\sum_{m=2}^{p-1}\left(\frac16m(m-1)(4m-5)+\frac12m(m-1)\right)\\
& & + \sum_{m=2}^p\left(\frac16(m-1)(m-2)(4m-3)+\frac12(m-1)(m-2)+(m-1)\right)\\
& = & \frac16(2p^4-6p^3+7p^2-3p)
\end{eqnarray*}
and
\begin{eqnarray*}
& &\sum_{m=2}^{p-1}\left(\frac16m(m-1)(4m-5)+\frac12m(m-1)\right)\\
& &+\sum_{m=2}^{p-1}\left(\frac16(m-1)(m-2)(4m-3)+\frac12(m-1)(m-2)+(m-1)\right)\\
&=& \frac16(2p^4-10p^3+19p^2-17p+6).
\end{eqnarray*}
Replacing $p=\frac12(k-1)$ in the first formula and $p=\frac k2$ in the second one we get
respectively
\begin{eqnarray*}
x_{\textrm{odd}} & = &
\frac1{48}\left(k^4-10k^3+38k^2-62k+33\right),\\
x_{\textrm{even}} & = & \frac1{48}\left(k^4-10k^3+38k^2-68k+48\right).
\end{eqnarray*}
Recalling that $\nu=5x+y$ and replacing the expressions just found for $x$ and $y=\binom{k-1}4$ we get
\begin{eqnarray*}
\nu_{\textrm{odd}} & = & \frac1{48}\left(7k^4-70k^3+260k^2-410k+213\right)\\
\nu_{\textrm{even}} & = & \frac1{48}\left(7k^4-70k^3+260k^2-440k+288\right).
\end{eqnarray*}
Thus
$$\nu=\nu_{\textrm{even}}+2\left(\frac k2-\left[\frac k2\right]\right)\left(\nu_{\textrm{odd}}-\nu_{\textrm{even}}\right)$$
and the stated formula easily follows.

\noindent
Dipartimento di Matematica\\
Universit\`a di Pisa\\
Largo Bruno Pontecorvo, 5\\
56127 PISA -- Italy\\
\texttt{petronio@dm.unipi.it}


\begin{thebibliography}{99}


\bibitem{Cohen}
P.~B.~Cohen (now P.~Tretkoff), \emph{Dessins d’enfant and Shimura varieties}, In: ``The Grothendieck
Theory of Dessins d’Enfants,'' (L.~Schneps, ed.), London Math. Soc. Lecture Notes Series, Vol. 200, Cambridge University Press, 1994,
pp.~237-243.

\bibitem{CoPeZa}    
P.~Corvaja --  C.~Petronio -- U.~Zannier, \emph{On certain
permutation groups and sums of two squares},
Elem. Math. \textbf{67} (2012), 169-181.

\bibitem{EKS}\textsc{A.~L.~Edmonds -- R.~S.~Kulkarni~ -- R.~E.~Stong},
\emph{Realizability of branched coverings of surfaces}, Trans.
Amer. Math. Soc. \textbf{282} (1984), 773-790.

\bibitem{GKL}
\textsc{I.~P.~Goulden -- J.~H.~Kwak -- J.~Lee},
\emph{Distributions of regular branched surface coverings},
European J. Combin. \textbf{25} (2004), 437-455.

\bibitem{Groth}
\textsc{A.~Grothendieck}, \textit{Esquisse d'un programme (1984)}. In: ``Geometric
Galois Action'' (L.~Schneps, P.~Lochak eds.), 1: ``Around Grothendieck's
Esquisse d'un Programme,'' London Math. Soc. Lecture Notes Series, Vol. 242,
Cambridge Univ. Press, 1997, pp. 5-48.

\bibitem{LZ}
\textsc{S.~K.~Lando -- A.~K.~Zvonkin},
``Graphs on Surfaces and their Applications,''
Encyclopaedia Math. Sci. Vol. 141, Springer, Berlin, 2004.

\bibitem{KM1}
\textsc{J.~H.~Kwak, A.~Mednykh},
\emph{Enumeration of branched coverings of closed orientable surfaces whose branch orders coincide
with multiplicity}, Studia Sci. Math. Hungar. \textbf{44} (2007), 215-223.

\bibitem{KM2}
\textsc{J.~H.~Kwak, A.~Mednykh},
\emph{Enumerating branched coverings over surfaces with boundaries},
European J. Combin. \textbf{25} (2004), 23-34.


\bibitem{KML}
\textsc{J.~H.~Kwak, A.~Mednykh, V.~Liskovets},
\emph{Enumeration of branched coverings of nonorientable surfaces with cyclic branch points},
SIAM J. Discrete Math. \textbf{19} (2005), 388-398.

\bibitem{Medn1}
\textsc{A.~D.~Mednykh},
\emph{On the solution of the Hurwitz problem on the number
of nonequivalent coverings over a compact Riemann surface} (Russian), Dokl. Akad. Nauk SSSR \textbf{261} (1981), 537-542.

\bibitem{Medn2}
\textsc{A.~D.~Mednykh},
\emph{Nonequivalent coverings of Riemann surfaces with a prescribed ramification type} (Russian),
Sibirsk. Mat. Zh. 25 (1984), 120-142.

\bibitem{MSS}
\textsc{S.~Monni, J.~S.~Song, Y.~S.~Song},
\emph{The Hurwitz enumeration problem of branched covers and Hodge integrals},
J. Geom. Phys. \textbf{50} (2004), 223-256.

\bibitem{Pako}\textsc{F.~Pakovich},
\emph{Solution of the Hurwitz problem for Laurent polynomials}, J. Knot Theory Ramifications \textbf{18} (2009), 271-302.

\bibitem{PaPe}
\textsc{M.~A.~Pascali -- C.~Petronio},
\emph{Surface branched covers and geometric $2$-orbifolds},
Trans. Amer. Math. Soc. \textbf{361}  (2009), 5885-5920

\bibitem{PaPebis}
\textsc{M.~A.~Pascali -- C.~Petronio}, \emph{Branched covers of the sphere
and the prime-degree conjecture},
Ann. Mat. Pura Appl.
\textbf{191} (2012), 563-594.

\bibitem{Bologna}
\textsc{E.~Pervova -- C.~Petronio},
{\it Realizability and exceptionality of candidate surface branched covers: methods and results},
Seminari di Geometria 2005-2009,
Universit\`a degli Studi di Bologna, Dipartimento di Matematica,
Bologna 2010, pp. 105-120.


\bibitem{x1}
\textsc{C.~Petronio},
{\it Explicit computation of some families of Hurwitz numbers},
preprint \texttt{arXiv:1805.00317}.

\bibitem{SongXu}
\textsc{J.~Song -- B.~Xu},
\emph{On rational functions with more than three branch points},
\texttt{arXiv:1510.06291}

\end{thebibliography}
\end{document}